\documentclass[a4paper,10pt]{article}
\usepackage[T1]{fontenc}
\usepackage{graphicx,amsmath,amsfonts,amssymb,amsthm,psfrag,color}

\newcommand{\eps}{\epsilon}

\newtheorem{thm}{Theorem}[section]
\newtheorem{corollary}[thm]{Corollary}
\newtheorem{lemma}[thm]{Lemma}
\newtheorem{prop}[thm]{Proposition}

\newtheorem{definition}[thm]{Definition}

\newtheorem{rem}[thm]{Remark}
\numberwithin{equation}{section}

\title{Statistics of transitions for Markov chains with periodic forcing.\footnote{supported by Conseil Regional de Bourgogne (contracts no. 2012-9201AAO047S01283 and no. 2012-9201AAO049S02781)}}
\author{S. Herrmann and D. Landon\\
Institut de Math\'ematiques de Bourgogne\\ UMR CNRS 5584,\\  
Universit\'e de Bourgogne,\\  
B.P. 47 870   21078 Dijon Cedex, France}
 \date{}
\begin{document}
%La version de TikZ est : \pgfversion
\maketitle
\begin{abstract} The influence of a time-periodic forcing on stochastic processes can essentially be emphasized in the large time behaviour of their paths. The statistics of transition in a simple Markov chain model permits to quantify this influence. In particular the first Floquet multiplier of the associated generating function can be explicitly computed and related to the equilibrium probability measure of an associated process in higher dimension. An application to the stochastic resonance is presented.  
\end{abstract}
\par\medskip

\par\medskip

\section*{Introduction}
The description of natural phenomenon sometimes requires to introduce stochastic models with periodic forcing. The simplest model used to interpret for instance the abrupt changes between cold and warm ages in paleoclimatic data is a one-dimensional diffusion process with time-periodic drift \cite{Ditlevsen}. This periodic forcing is directly related to the variation of the solar constant (Milankovitch cycles). In the neuroscience framework, such periodic forced model is also of prime importance: the firing of a single neuron stimulated by a periodic input signal can be represented by the first passage time of a periodically driven Ornstein-Uhlenbeck process \cite{Plesser} or other extended models \cite{longtin}. Moreover let us note that seasonal autoregressive moving average  models have been introduced in order to analyse and forecast statistical times series with periodic forcing. Recently the time dependence of the volatility in financial time series leaded to emphasize periodic autoregressive conditional 
heteroscedastic models. Whereas several statistical models permit to deal with time series, the influence of periodic forcing on time-continuous stochastic processes concerns only few mathematical studies. Let us note a nice reference in the physics literature dealing with this research subject \cite{Jung}. 

Therefore we propose to study a simple Markov chain model evolving in a time-periodic environment (already introduced in the stochastic resonance context \cite{imk-pav} and \cite{herr-imk}) and in particular to focus our attention to its large time asymptotic behaviour. Since the dynamics of the Markov chain is not time-homogeneous, the classical convergence towards the invariant measure and the related convergence rate cannot be used. One essential tool is to increase the dimension of the state space in order to construct an appropriate homogeneous Markov chain and to apply classical ergodic results. 

\emph{Description of the model.} Let us consider a time-continuous Markov chain evolving in the state space $\mathcal{S}=\{-1,1\}$ whose transition rates correspond to $\varphi_+^0$ respectively $\varphi_+^0$, the exit rate of the state $-1$ resp. $+1$. We assume that $\varphi_\pm^0>0$ and we perturb this initial process by a periodic forcing of period $T$; it means that the transition rates are increased using two additional non negative periodic functions $\varphi_\pm^p$. The obtained Markov chain is denoted by  $(X_t)_{t\ge 0}$ and its infinitesimal generator is given by
\begin{equation}
\label{eq:gen-inf}
Q_t=\left(\begin{array}{cc}
-\varphi_-(t) & \varphi_+(t)\\
\varphi_-(t) & -\varphi_+(t)
\end{array}\right),
\end{equation}
where $\varphi_\pm(t)=\varphi_\pm^0+\varphi_\pm^p(t)$ are $T$-periodic functions. In order to describe precisely the paths of the chain $(X_t)$, we define transitions statistics: $\mathcal{N}_t$ corresponds to the number of switching from state $-1$ to $+1$ up to time $t$. The law of the random process $\mathcal{N}_t$ at a fixed time $t$ is characterized by its moment generating function:
\begin{equation}
\label{intro:def:psi}
\Psi(\eta,t):=\mathbb{E}[\eta^{\mathcal{N}_t}], \quad \eta>1.
\end{equation}

\emph{Main result.} Let us first note that, in the higher dimensional space $[0,T]\times\mathcal{S}$, we can define a Markov process $(t\ {\rm mod}\ T, X_t)_{t\ge 0}$ which is time-homogeneous and admits a unique invariant measure $\mu=(\mu_-(t),\mu_+(t))_{0\le t<T}$. The main result can then be stated. The periodic forcing implies the use of Floquet's theory to obtain a precise description of the moment generating function: there exist two time-periodic functions $p_{i}(\eta,t)$ with $i=1,2$ such that
\begin{equation}
\label{eq:main-res}
\Psi(\eta,t)=p_1(\eta,t)e^{\lambda_1 t}+p_2(\eta,t)e^{\lambda_2 t}.
\end{equation}
where $p_1$ is a positive function, $\lambda_1>1>\lambda_2$ and
\begin{equation}
\label{intro:main}
\lambda_1=\frac{\log \eta}{T}\ \int_0^T\varphi_-(s)\mu_-(s)ds,\quad \lambda_2 =\frac{1}{T}\int_0^T\varphi_-(s)+\varphi_+(s)\ ds-\lambda_1.
\end{equation}
This result implies in particular that, if the observed time-interval $[0,t]$ is increased by adding a period, the associated function $\Psi(\eta,t)$ is multiplied by a parameter which becomes close to 
\[
\eta^{\int_0^T\varphi_-(s)\mu_-(s)ds}
\]
as $t$ becomes large.

\emph{Application.} The explicit expression \eqref{intro:main} of the first Floquet exponent $\lambda_1$ permits to deal with particular optimization problems appearing in the stochastic resonance framework (see, for instance, \cite{gam}). Let us consider a family of periodic forcing having all the same period $T$ and being parametrized by a variable $\epsilon$, then it is possible to choose in this family the perturbation which has the most influence on the stochastic process, just by minimizing the following \emph{quality measure}:
\[
\mathcal{M}(\epsilon):=\left\vert \int_0^T\varphi_-^\epsilon(s)\mu_-^\epsilon(s)ds-1 \right\vert.
\]
In Section \ref{sec:resonance} we shall compare this quality measure (already introduced in \cite{Talkner}) to other measures usually used in the physics literature \cite{imk-pav}. 

\section{Periodic stationary measure for Markov chains}
\label{sec:1}
Before focusing our attention to the paths behaviour of the Markov chain, we describe, in this preliminary section, the fixed time distribution of the random process and, in particular, analyse the existence of a so-called \emph{periodic stationary probability measure -- PSPM} (we shall precise this terminology in the following). \\
The distribution of the Markov chain $(X_t)_{t \ge 0}$ starting from the initial distribution $\nu_0$ and evolving in the state space $\{-1,1\}$ is characterized by 
\[
\nu_\pm(t)=\mathbb{P}_{\nu_0}(X_t=\pm 1).
\]
This probability measure $\nu=(\nu_-,\nu_+)^*$ (the symbol $*$ stands for the transpose) constitutes a solution to the following ode: 
\begin{equation}
\label{eq:prob}
\frac{d\nu(t)}{dt}=Q_t\nu(t) \quad\mbox{and}\quad \nu(0)=\nu_0,
\end{equation}
where the generator $Q_t$ is defined in \eqref{eq:gen-inf}.\\ Floquet's theory dealing with linear differential equation with periodic coefficients can thus be applied. In particular we shall prove that $\nu(t)$ converges exponentially fast towards a periodic solution of \eqref{eq:prob}, the convergence rate being related to the Floquet multipliers (see Section 2.4 in \cite{Chicone}). 
\begin{definition} Any $T$-periodic solution $\nu(t)=(\nu_-(t),\nu_+(t))^*$ of \eqref{eq:prob} is called a \emph{periodic stationary probability measure -- PSPM} iif $\nu_\pm(t)>0$ and $\nu_-(t)+\nu_+(t)=1$ both for all $t\ge 0$.
\end{definition}
\noindent The following statement points out the long time asymptotics of the Markov chain.
\begin{prop}\label{prop:largetime} In the large time limit, the probability distribution $\nu$ converges towards the unique PSPM $\mu$ defined by
$\mu(t)=(\mu_-(t),1-\mu_-(t))$
and
\begin{equation} \label{eq:loi_inva:01}
\mu_-(t)=\mu_-(0) e^{-\int_0^t (\varphi_-+\varphi_+)(s)ds}+\int_0^t\varphi_+(s)\ e^{-\int_s^t (\varphi_-+\varphi_+)(u)du}ds, 
\end{equation}
where
\begin{equation} \label{eq:loi_inva:02}
\mu_-(0)=\frac{I(\varphi_+)}{I(\varphi_-+\varphi_+)}\quad\mbox{and}\quad
I(f)=\int_0^Tf(t)e^{-\int_t^T (\varphi_-+\varphi_+)(u)du}dt. 
\end{equation}
More precisely:
\begin{equation}\label{eq:loi_inva:03}
\lim_{t\to\infty}\frac{1}{t}\log\Vert  \nu(t)-\mu(t) \Vert\le\lambda_2,
\end{equation}
where $\lambda_2$ stands for the second Floquet exponent:
\begin{equation}\label{eq:loi_inva:04}
\lambda_2=-\frac{1}{T}\ \int_0^T(\varphi_-+\varphi_+)(t)\,dt.
\end{equation}
\end{prop}
\begin{rem} It is possible to transform $(X_t)$ into a time-homogeneous Markov process just by increasing the space dimension. By this procedure $(\mu(t))_{0\le t<T}$ becomes the invariant probability measure of  $(t\ \operatorname{mod}\ T,X_t)_{t\ge 0}$.
\end{rem}
\begin{proof} 1. First we study the existence of a unique PSPM. Let $\mu(t)$ be a probability measure  thus $\mu_-(t)+\mu_+(t)=1$. If $\mu$ satisfies \eqref{eq:prob} then we obtain, by substitution, the differential equation:
\[
\frac{d}{dt}\mu_-(t)=-\varphi_-(t)\mu_-(t)+ \varphi_+(t)(1-\mu_-(t)).
\]
This equation can be solved using the variation of the constant. The procedure yields to \eqref{eq:loi_inva:01}. The periodicity of the solution requires $\mu_-(T)=\mu_-(0)$ and leads to \eqref{eq:loi_inva:02}.\\
2. The system \eqref{eq:prob} admits two Floquet multipliers $\rho_1$ and $\rho_2$. Since there exists a periodic solution, one of the multipliers (let's say $\rho_1$) is equal to $1$ and we can compute the other one using the relation between the product $\rho_1\rho_2$ and the trace of $Q_t$:
\[
\rho_1\rho_2=\exp\left(\int_0^T {\rm tr}(Q_t)\,dt\right).
\]
The explicit expression of the trace leads to \eqref{eq:loi_inva:04}. Let us just note that 
we can link to both Floquet multipliers $\rho_1$ and $\rho_1$  the so-called Floquet exponents $\lambda_1$ and $\lambda_2$ defined (not uniquely) by
 $$\rho_1=e^{\lambda_1 T} \quad \text{and} \quad \rho_2=e^{\lambda_2 T}.$$
3. Since the Floquet multipliers are different, each multiplier is associated with a particular solution  of \eqref{eq:prob}. $\rho_1=1$ (i.e. $\lambda_1=0$) corresponds to the PSPM since $\mu(t+T)=\rho_1\mu(t)$ for all $t\in\mathbb{R}_+$. For the Floquet exponent $\lambda_2$, we consider $\rho(t)$ the solution of  \eqref{eq:prob} with initial condition  $\rho(0)^*=(-1,1)$. Combining both equations of \eqref{eq:prob}, we obtain
\begin{equation}
\label{eq:second}
\left\{\begin{array}{ll}
\rho_-(t)+\rho_+(t)=0\\
\rho_-(t)-\rho_+(t)=-2\exp-\int_0^t(\varphi_-+\varphi_+)(s)ds.
\end{array}\right.
\end{equation}
We deduce
\[
\rho(t)^*=\left(-\exp-\int_0^t(\varphi_-+\varphi_+)(s)ds,\ \exp-\int_0^t(\varphi_-+\varphi_+)(s)ds \right)
\]
and we can easily check that $\rho(t+T)=\rho(t)e^{\lambda_2 T}$.\\
The solution of \eqref{eq:prob} with any initial condition is therefore a linear combination of $\rho$ and $\mu$, the solutions corresponding to the Floquet multipliers. Writing $\nu(0)$ in the basis $(\mu(0),\rho(0))$ yields
\(
\nu(t)=\alpha \mu(t)+\beta \rho(t),
\)
with $\alpha=\nu_+(0)+\nu_-(0)$ (equal to $1$ in the particular probability measure case) and 
\[
\beta=\frac{\nu_+(0)-\nu_-(0)}{2}+\alpha\frac{I(\varphi_+)-I(\varphi_-)}{2I(\varphi_-+\varphi_+)}.
\]
Then, if the initial condition is a probability, we obtain \eqref{eq:loi_inva:04} since
\begin{align*}
\Vert \nu(t)-\mu(t) \Vert=\Vert \beta \rho(t)\Vert=\sqrt{2}\vert \beta\vert e^{-\int_0^t(\varphi_-+\varphi_+)(s)ds}. 
\end{align*}
\end{proof}
\section{Statistics of transitions}
In the previous section, the study of the process points out how fast its distribution converges towards a periodic stationary distribution (in the sense that $X_t$ and $X_{t+T}$ are identically distributed). The aim now is to improve this result, which is just related to the position of the chain at some fixed time $t$, by analysing the paths behaviour in the large time scale, especially the statistics of the transitions between the two states namely $-1$ and $+1$. That's why we introduce the moment generating function
\begin{equation}
\label{eq:def:gen}
\psi(\eta,t)=\mathbb{E}[\eta^{\mathcal{N}_t}]
\end{equation}
associated with $\mathcal{N}_t$, the number of  transitions from state $-1$ to state $1$ up to  time $t$. We can decompose this generating function into two parts:
\[
\psi(\eta,t)=\psi_-(\eta,t)+\psi_+(\eta,t)\quad\mbox{where}\quad \psi_\pm(\eta,t)=\mathbb{E}\Big[\eta^{\mathcal{N}_t}1_{\{X_t=\pm 1\}}\Big].
\]
Then the vector $\Psi(\eta,t)=(\psi_-(\eta,t),\psi_+(\eta,t))^*$ satisfies the ode:
\begin{equation}
\label{eq:dif:gene}
\frac{\partial \Psi(\eta,t)}{\partial t}=Q(\eta,t)\Psi(\eta,t),\quad\mbox{with}\quad Q(\eta,t)=\left(\begin{array}{cc}
-\varphi_-(t) & \varphi_+(t)\\
\eta\varphi_-(t) & -\varphi_+(t)
\end{array}\right).
\end{equation}
Since \eqref{eq:dif:gene} is a differential equation with periodic coefficients, Floquet's theory can be applied and, in this way, the Floquet multipliers describe the large time behaviour of any solution to the equation \eqref{eq:dif:gene}, in particular the generating function. That's why we shall compute explicitly these multipliers. In Section \ref{sec:1} the system of ode considered has been reduced to a one-dimensional equation, here it is not the case so that we need to introduce an other procedure: a time-discretization approach. 
\mathversion{bold}
\subsection{The discrete-time counterpart of the chain $(X_t)_{t\in\mathbb{R}_+}$}
\label{sec:model-discret}
\mathversion{normal}
We consider a non-homogeneous Markov chain in discrete time $(Z_n^N,\ n\in\mathbb{N})$ defined on the space state $\mathcal{S}=\{-1,1\}$. The associated transition matrix is given by
\[
\Pi^N_n=\left(\begin{array}{cc}
1-\pi_{n,N}^- & \pi_{n,N}^+\\
\pi_{n,N}^- & 1-\pi_{n,N}^+
\end{array}\right),
\]
where $(\pi_{n,N}^\pm)_{n\ge 0}$ are periodic sequences of period $N$. Let us introduce different quantities related to $(\pi_{n,N}^\pm)$:
\begin{equation}
\label{eq:def:alpha}
\alpha_n^N:=1-(\pi_{n,N}^++\pi_{n,N}^-)\quad\mbox{and}\quad A_k^N:=1_{\{k\ge N\}}+1_{\{0\le k\le N-1\}}\prod_{j=k}^{N-1}\alpha_j^N. 
\end{equation}
For notational simplicity, the index $N$ shall voluntarily be removed when there is no ambiguity. Let us now consider usual properties of the discrete-time Markov chain: existence of \emph{periodic stationary probability measure}, uniqueness and ergodic properties.
\begin{prop}\label{prop:mes-inv-disc}
The Markov chain $(Z_n^N,\ n\ {\rm mod}\ N)_{n\ge 0}$ admits a unique stationary probability measure $\mu^N$ given by
\begin{equation}
\label{eq:prop:1}
\mu^N(-1,n)=\dfrac{A_0^N}{A_{n}^N}\mu^N(-1,0)+\sum_{k=0}^{n-1}\pi_{k,N}^+ \dfrac{A_{k+1}^N}{A_n^N}
\end{equation}
and
\begin{equation}
\label{eq:prop:2}
\mu^N(-1,0)=\frac{\sum_{k=0}^{N-1}\pi_{k,N}^+A_{k+1}^N}{1-A_0^N}.
\end{equation}
\end{prop}
\begin{rem} Since the PSPM is a probability, we compute $\mu^{N}(-1,n)$ and deduce $\mu^N(1,n)$. 
%Thus, for example,
%\[
%\mu^N(1,0)=\frac{1-A_0^N - \sum_{k=0}^{N-1}\pi_{k,N}^+A_{k+1}^N}{1-A_0^N}.
%\]
\end{rem}
\begin{proof}
Let us define $\nu_n$ the probability that the chain $Z_n$ is in the state $-1$ at time $n$. The chain is initialized through $\nu_0=\mathbb{P}(Z_0=-1)$. Then applying the matrix $\Pi_0$, the distribution of $Z_1$ satisfies
\(
\nu_1=\alpha_0\nu_0+\pi_0^+.
\)
The same argument yields the distribution of $Z_2$:
\[
\nu_2=\alpha_1\nu_1+\pi_1^+=\alpha_1\alpha_0\nu_0+\alpha_1\pi_0^++\pi_1^+.
\]
By induction, we obtain the general formula:
$$\nu_n = \left( \prod_{j=0}^{n-1} \alpha_i \right) \nu_0 + \sum_{k=0}^{n-1} \left( \pi_k^+ \prod_{j=k+1}^{n-1}\alpha_j \right)$$ 
which can be rewritten by use of the quantities $A_k^n$ as follows:
\begin{equation}
\label{eq:formu:gen:dem}
\nu_n=\dfrac{A_0^N}{A_{n}^N} \nu_0+\sum_{k=0}^{n-1}\pi_k^+ \dfrac{A_{k+1}^{N}}{A_{n}^{N}},\quad n\ge 0.
\end{equation}
Since the Markov chain $(Z_n^N,\ n\ {\rm mod}\ N)_{n\ge 0}$ is positive and recurrent, there exists a unique stationary probability measure $\mu$. This measure is obtained by setting $\mu(-1,n)=\nu_n$ and solving  $\nu_N=\nu_0$. This leads immediately to the value of $\nu_0$ and in the sequel $\nu_n$ for any $1\le n<N$ using \eqref{eq:formu:gen:dem}.
\end{proof}
\begin{thm}
\label{thm:asymp:disc}
Let $\mathcal{N}_n^N$ be the number of transitions from state $-1$ to state $+1$ performed by the Markov chain $(Z_n^N)_{n\ge 0}$ up to time $n$ (included). We denote by $\Psi^N(\eta,n)$ its moment generating function:
\begin{equation}
\label{eq:def:fct}
\Psi^N(\eta,n):=\mathbb{E}[\eta^{\mathcal{N}_n^N}].
\end{equation}
The following asymptotic behaviour holds:
\begin{equation}
\label{eq:thm:gen-disc}
\lim_{n\to\infty}\frac{1}{n}\log\Psi^N(\eta,n)=\frac{\log(\eta)}{N}\sum_{k=0}^{N-1}\pi_{k,N}^-\mu^N(-1,k)=\frac{\log(\eta)}{N}\,
\mathbb{E}_{\mu^N}[\mathcal{N}_N^N],
\end{equation}
where $\mu^N$ is the stationary distribution of the Markov chain $(Z_n^N,n\ {\rm mod} \ N)_{n\in\mathbb{N}}$ defined in Proposition \ref{prop:mes-inv-disc}.
\end{thm}
\begin{proof}{\bf Step. 1} Let us prove that for any $j> 0$
\begin{equation}\label{eq:recu}
\mathbb{E}_{\mu^N}[\mathcal{N}_j^N]=\sum_{k=0}^{j-1}\pi_{k,N}^-\mu^N(-1,k).
\end{equation}
The second equality in \eqref{eq:thm:gen-disc} is then immediate. We obtain \eqref{eq:recu} by induction: for $j=1$, we easily have
\[
\mathbb{E}_{\mu^N}[\mathcal{N}_1^N]=\mathbb{P}_{\mu^N}(Z_0=-1,\, Z_1=1)=\mu^N(-1,0)\pi^-_{0,N}.
\]
We assume that \eqref{eq:recu} is satisfied for some $j\ge 1$, then
\[
\mathbb{E}_{\mu^N}[\mathcal{N}_{j+1}^N]=\mathbb{E}_{\mu^N}[\mathcal{N}_j^N+1_{\{ Z_j=-1,\, Z_{j+1}=1 \}}]=\mathbb{E}_{\mu^N}[\mathcal{N}_j^N]+\mu^N(-1,j)\pi^-_{j,N}.
\]
This leads to \eqref{eq:recu} with $j$ replaced by $j+1$. Formula \eqref{eq:recu} is therefore satisfied for all $j$ especially for $j=N$.\\
{\bf Step 2.} The ergodic theorem is the essential tool for studying the long time behaviour. First of all we construct a new $\{-1,1\}^2$-valued Markov chain $(Y_n)_{n\ge 1}$ defined by
\[
Y_n=(Z_{n-1},Z_n).
\]
The associated transition matrix depends on $\pi_n^\pm$ and $n$; moreover it is $N$-periodic. To deal with an homogeneous chain, it suffices to consider $(Y_n,n\ {\rm mod}\ N)$. Suppose that the law of $Z_n$ is the invariant periodic measure describe in Proposition \ref{prop:mes-inv-disc} then $Y_n$ is also in the invariant regime: for all $(a,b)\in\{-1,1\}^2$, we have
\[
\mathbb{P}(Y_n=(a,b))=\mathbb{P}(Z_{n-1}=a,\,Z_n=b)=\mu(a,n-1)\mathbb{P}(Z_n=b\vert Z_{n-1}=a).
\]
In the previous expression, both terms constituting the right hand side are $N$-periodic. Hence $\mathbb{P}(Y_n=(a,b))$ is periodic. Let us therefore define the following measure $\mu^Y$:
\begin{eqnarray*}\left\{\begin{array}{l}
\mu^Y(-1,-1,n)=\mu(-1,n-1)(1-\pi^-_{n-1})\\
\mu^Y(-1,+1,n)=\mu(-1,n-1)\pi^-_{n-1}\\
\mu^Y(+1,-1,n)=\mu(+1,n-1)\pi^+_{n-1}\\
\mu^Y(+1,+1,n)=\mu(+1,n-1)(1-\pi^+_{n-1}).\end{array}\right.
\end{eqnarray*}
Let us just observe that $\mu^Y$ is a positive invariant measure for the homogeneous Markov chain $(Y_n,\,n\ {\rm mod}\ N)$. It suffices to normalize the measure in order to obtain an invariant probability: $\frac{1}{N}\, \mu^Y$.
Let us now consider the moment generating function $\Psi(\eta,n)$ associated with the number of transitions $\mathcal{N}_n$. By construction, $\mathcal{N}_n$ is also the number of time the   non-homogeneous Markov chain visits the state $(-1,1)$ or finally the number of times the homogeneous Markov chain $(Y_n,\,n\ {\rm mod}\ N)$ visits the set 
\[
\mathcal{A}=\Big\{ (-1,1,1),(-1,1,2),\ldots,(-1,1,N) \Big\}.
\]
Since the homogeneous chain is recurrent positive and aperiodic, the ergodic theorem implies
\[
\lim_{n\to\infty}\frac{\mathcal{N}_n}{n}=\frac{1}{N}\mu^Y(\mathcal{A})=\frac{1}{N}\sum_{k=0}^{N-1}\pi_{k}^-\mu(-1,k)\quad \mbox{a.s.}
\] 
More precisely the law of the iterated logarithm implies the existence of a constant $\gamma>0$ such that 
\begin{equation}\label{eq:lil}
\limsup_{n\to\infty}\frac{1}{\sqrt{n\log\log n}}\left| \mathcal{N}_n-\frac{n}{N}\mu^Y(\mathcal{A}) \right|\le \gamma\quad \mbox{a.s.}
\end{equation}
Let us then define the event $\Omega_n$: the set of all paths satisfying
\[
\frac{1}{\sqrt{n\log\log n}}\left| \mathcal{N}_n-\frac{n}{N}\mu^Y(\mathcal{A}) \right|\le \gamma.
\]
Hence, for $\eta\le 1$, we obtain the associated decomposition
\begin{align*}
\Psi(\eta,n)=\mathbb{E}[\eta^{\mathcal{N}_n}1_{\Omega_n}]+\mathbb{E}[\eta^{\mathcal{N}_n}1_{\Omega_n^c}].
\end{align*}
Using Lebesgue's dominated convergence theorem and \eqref{eq:lil}, the second term converges to $0$ when $n$ tends to infinity. Moreover
\[
\mathbb{P}(\Omega_n)\exp\Big\{\Big(\frac{n}{N}\mu^Y(\mathcal{A})-\gamma\sqrt{n\log\log n}\Big)\log(\eta)\Big\}\le\mathbb{E}[\eta^{\mathcal{N}_n}1_{\Omega_n}]
\]
and by symmetry
\[
\mathbb{E}[\eta^{\mathcal{N}_n}1_{\Omega_n}]\le \mathbb{P}(\Omega_n)\exp\Big\{\Big(\frac{n}{N}\mu^Y(\mathcal{A})+\gamma\sqrt{n\log\log n}\Big)\log(\eta)\Big\}.
\]
Since $\mathbb{P}(\Omega_n)$ converges to $1$ as $n\to\infty$, the combination of both previous inequalities leads to
\[
\lim_{n\to\infty}\frac{1}{n}\log\Psi^N(\eta,n)=\log(\eta)\frac{1}{N}\mu^Y(\mathcal{A}).
\]
In order to conclude the proof, it suffices to compute the explicit expression of $\mu^Y(\mathcal{A})$.
\end{proof}
In a similar way to Section \ref{sec:1}, the asymptotic behaviour of the moment generating function is related to the eigenvalues of a suitable matrix. Indeed, if we decompose the generating function as follows
\begin{equation}\label{eq:def:psi+}
\Psi^N(\eta,n)=\Psi^N_-(\eta,n)+\Psi^N_+(\eta,n)
\quad\mbox{with}\quad 
\Psi^N_\pm(\eta,n)=\mathbb{E}[\eta^{\mathcal{N}_n^N}1_{\{ Z_n^N=\pm 1 \}}],
\end{equation}
then the vector ${\bf \Psi}(n):=(\Psi^N_-(\eta,n),\Psi^N_+(\eta,n))^*$ satisfies the recurrence relation:
\begin{equation}
\label{eq:matrice-gen-disc}
{\bf \Psi}(n+1)=M_n^N {\bf \Psi}(n)
\quad\mbox{where}\quad
M_n^N=\left(\begin{array}{cc}
1-\pi_{n,N}^- & \pi_{n,N}^+\\
\eta\pi_{n,N}^- & 1-\pi_{n,N}^+
\end{array}
  \right).
\end{equation}
Let us define the product of matrices 
\begin{equation}\label{eq:def:prodm}
\mathcal{M}^N=M_{N-1}^N M_{N-2}^N \ldots M^N_0.
\end{equation}
We thus obtain ${\bf \Psi}(N)=\mathcal{M}^N {\bf \Psi}(0)$. We observe in the following statement the link between the asymptotic behaviour of the moment generating function and the eigenvalues of the monodromy matrix.
\begin{corollary}\label{prop:valp} Let $\eta>1$. The eigenvalues of the monodromy matrix $\mathcal{M}^N$ satisfy: $\lambda_1^N>1>\lambda_2^N$ and 
\begin{equation}\label{eq:prop}
\log(\lambda_1^N)=\log(\eta)\sum_{k=0}^{N-1}\pi_{k,N}^-\mu^N(-1,k).
\end{equation}
\end{corollary}
\begin{proof}  The matrix $\mathcal{M}^N$ is directly related to the fundamental solution of equation \eqref{eq:matrice-gen-disc}. In fact its coefficients correspond to
\[
\mathbb{E}[\eta^{\mathcal{N}_N^N}1_{\{ Z_N^N=a \}}|Z_0^N=b]\quad\mbox{with}\quad (a,b)\in\{-1,1\}^2.
\]
These coefficients are all positive: the matrix $\mathcal{M}^N$ thus admits two distinct real eigenvalues, denoted by $\lambda_1^N$ and $\lambda_2^N$ ($\lambda_1^N$ corresponds to the largest one), and is diagonalizable. We denote by ${\bf u_i}$, $i=1,2$ the associated eigenvectors. Hence ${\bf \Psi}(0)$ can be expressed in this basis: there exists $r_1$ and $r_2$ such that ${\bf \Psi}(0)=r_1{\bf u_1}+r_2{\bf u_2}$. We immediately deduce
\begin{equation}\label{eq:psidev}
{\bf \Psi}(kN)=r_1(\lambda_1^N)^k{\bf u_1}+r_2(\lambda_2^N)^k{\bf u_2},\quad\forall k\ge 0.
\end{equation}
Let us prove that $\lambda_1^N>1>\lambda_2^N$. If we compute the determinant of $M_n^N$, we obtain:
\[
{\rm det}(M_n^N)=1-\pi_n^+-\pi_n^-+(1-\eta)\pi_n^+\pi_n^-=\alpha_n+(1-\eta)\pi^+_n\pi_n^-.
\]
We recall that $\alpha_n$ is defined in \eqref{eq:def:alpha}. Since $\log(\lambda_1^N\lambda_2^N)=\log({\rm det}(\mathcal{M}^N))$, 
\begin{equation}\label{eq:prodvp}
\log(\lambda_1^N\lambda_2^N)=\sum_{k=0}^{N-1}\log({\rm det}(M_k^N))=\sum_{k=0}^{N-1}\log\Big(\alpha_k+(1-\eta)\pi^+_k\pi_k^-\Big).
\end{equation}
Due to the assumption $\eta>1$, \eqref{eq:prodvp} leads to $\log(\lambda_1^N\lambda_2^N)<0$. Suppose that  $\log(\lambda_1^N)\le 0$ (consequently $\log(\lambda_2^N)<0$), then according to \eqref{eq:psidev}, $({\bf \Psi}(kN))_{k\ge 0}$ is a bounded sequence. Of course, this is a nonsense since ${\bf \Psi}(kN)$ is the generating function (associated with $\eta>1$) of a growing process $\mathcal{N}$ which a.s. tends  to infinity. Consequently $\log(\lambda_1^N)> 0$, $\log(\lambda_2^N)< 0$ and $r_1\neq 0$.
Finally we deduce the large time asymptotic behaviour:
\begin{equation}
\label{eq:comp:asym}
\lim_{k\to\infty}\frac{1}{kN} \log\langle {\bf \Psi}(kN),{\bf 1}\rangle=\frac{\log(\lambda_1^N)}{N}
\end{equation}
where ${\bf 1}=(1,1)^*$. Using \eqref{eq:def:psi+} and  Theorem \ref{thm:asymp:disc}, we obtain \eqref{eq:prop}.
\end{proof}
\begin{rem} It is also possible to compute the second eigenvalue of the monodromy matrix $\mathcal{M}^N$. Indeed, by \eqref{eq:prodvp}, we get:
\[
\log(\lambda_2^N)=\sum_{k=0}^{N-1}\log\Big(\alpha_k+(1-\eta)\pi^+_k\pi_k^-\Big)-\log(\eta)\sum_{k=0}^{N-1}\pi_{k,N}^-\mu^N(-1,k).
\]
\end{rem}

\subsection{On the convergence from discrete to continuous time}
In the previous section, the asymptotic behaviour of a periodic discrete-time Markov chain was emphasized. The study of such Markov chain was a first step in the analysis of continuous-time Markov chain. We shall now describe how all previous results can have a continuous counterpart.\\ 
In this section, we especially prove that the Markov chain $Z_n^N$, the expression $A_k^N$ defined by \eqref{eq:def:alpha}, the probability measures $\mu^N(-1,n)$ defined by \eqref{eq:prop:1} and \eqref{eq:prop:2} and finally the moment generating function 
$\psi^N(\eta,n)$ converge as $N$ becomes large.\\[5pt]
Let us assume that the transition probabilities of the Markov chain $Z_n^N$ are small with respect to $1/N$:
\begin{equation}\label{eq:hypo}
 \pi_{j,N}^- = \dfrac{T}{N}\ \varphi_- \left( \dfrac{j}{N} T \right) \quad \text{and} \quad \pi_{j,N}^+ = \dfrac{T}{N}\ \varphi_+ \left( \dfrac{j}{N} T \right) \quad \text{for all} \quad j \in \mathbb{N}
\end{equation}
where $\varphi_-$ and $\varphi_+$ are piecewise continuous functions defined in \eqref{eq:gen-inf}.
\begin{lemma} \label{lem:convA}
Let $t\geq 0$. Under the assumption \eqref{eq:hypo}, the following convergence holds uniformly w.r.t. the variable $t$ 
\begin{equation} \label{eq:convAk}
 \lim_{N\to\infty}A_{\left\lfloor \frac{tN}{T} \right\rfloor}^N =
 \begin{cases}
A_t:=\exp \left( - \int_t^T (\varphi_-+\varphi_+)(s) ds \right) & \text{if} \quad 0\le t < T \\
1 & \text{otherwise.}  
\end{cases}
\end{equation}
We recall that $A^N_\cdot$ is defined in \eqref{eq:def:alpha}.
\end{lemma}
\begin{proof} For notational simplicity, we set $k = \lfloor \frac{t}{T}N \rfloor$.
By definition  $A_k^N=1$ for $k\ge N$. Using \eqref{eq:hypo}, we obtain for $0\le k\le N-1$,
$$\log A_k^N = \sum_{j=k}^{N-1} \log \left( 1 - \dfrac{T}{N} (\varphi_-+\varphi_+)\left(\dfrac{jT}{N}\right) \right).$$
Using the Taylor expansion, we get  
$$\log A_k^N = \dfrac{T}{N}  \sum_{j=k}^{N-1} \left[ -(\varphi_-+\varphi_+)\left(\dfrac{jT}{N}\right) + \mathcal{O} \left( \dfrac{1}{N}  \right) \right].$$
To conclude it suffices to use the uniform convergence of the Riemann series theorem on one side and to prove that the error term converges uniformly towards $0$. The details are left to the reader. 
%We have the uniform convergence 
%$$\lim_{N \to \infty} \dfrac{T}{N} \sum_{j=k}^{N-1} \dfrac{1}{N}  = 0$$
%and the first part is Riemann sum which
%converge uniformly in time on $[0,T]$ to 
%$$ -\int_t^T (\varphi_-+\varphi_+)(s) ds.$$
\end{proof}

Since we have a link between the transition probabilities of the discrete-time Markov chain (through $A_\cdot^N$) and the transition probabilities of the continuous one $\varphi_\pm$, we shall compare the processes themselves. First we investigate the comparison of the stationary measures and secondly we point out the convergence in law of the processes.\\
Let us define the time-continuous process $(Z^N(t))_{t \geq 0}$ associated with the discrete Markov chain $(Z_k^N,k \mod N)_{k \in \mathbb{N}}$
as follows
\begin{equation} \label{eq:def:processuscontinu}
Z^N(t)=Z_k^N, \quad\mbox{for}\quad k\dfrac{T}{N} \le t < (k+1) \dfrac{T}{N}, \quad k \in \mathbb{N}.
\end{equation}
Let us note that $(Z^N(t))_{t \geq 0}$ is a piecewise constant and c\`adl\`ag process. 
\begin{prop} \label{prop:conv:mesinv}
The stationary probability measure $\mu^N(t)$ associated with the process $(Z^N(t), t \mod T)_{t \geq 0}$ converges to the stationary distribution $\mu(t)$ of the Markov chain $(X_t,t \mod T)_{t \ge 0}$. This convergence holds uniformly w.r.t. the time variable $t$. 
\end{prop} 

\begin{rem}
We shall use the following result : if $B_{\lfloor \frac{tN}{T} \rfloor}^N$ converges uniformly to $B(t)$ on $[0,T]$ and if $\phi$ is piecewise continuous on $[0,T]$, then the following convergence holds uniformly
$$\dfrac{T}{N} \sum_{j=0}^{N-1} \phi \left( \dfrac{jT}{N} \right) B_j^N\to \int_0^T \phi(s) B(s) ds.$$ 
\end{rem}

\noindent\emph{Proof of Proposition \ref{prop:conv:mesinv}.} 
We first focus our attention to the convergence of $\mu^N(-1,0)$ given by \eqref{eq:prop:2}. The assumption \eqref{eq:hypo} leads to
\begin{equation} \label{eq:preuve:conv:mesinv:1}
\sum_{k=0}^{N-1} \pi_{k,N}^+ A_{k+1}^N = \dfrac{T}{N} \sum_{k=0}^{N-1}  \varphi_+ \left( \dfrac{kT}{N} \right) A_{k+1}^N
\end{equation}
which converges uniformly (Lemma \ref{lem:convA}) towards
\begin{equation}\label{eq:ajj}
\int_0^T \varphi_+(s) e^{-\int_s^T (\varphi_-+\varphi_+)(u) du} ds = I(\varphi_+).
\end{equation}
Let us just recall that $I(\cdot)$ has been defined in \eqref{eq:loi_inva:02}.
Using \eqref{eq:prop:2}, \eqref{eq:preuve:conv:mesinv:1}, \eqref{eq:ajj} and the identity $1-A_0= I(\varphi_-+\varphi_+)$, leads to 
 \begin{equation} \label{eq:preuve:conv:mesinv:3}
 \lim_{N \to \infty} \mu^N(-1,0) = \mu_-(0).
 \end{equation}
Using similar arguments, we prove the  convergence of $\mu^N(-1,k)$ defined by \eqref{eq:prop:1}. Let $t$ and $s$ two real numbers such that 
$t \ge s \ge 0$. 
We define $k$ and $j$ by
$k = \lfloor \frac{t}{T}N \rfloor$ and $j = \lfloor \frac{s}{T}N \rfloor$.
We observe the following uniform convergences:
\begin{equation} \label{eq:preuve:conv:mesinv:4}
\lim_{N \to \infty} \sum_{j=0}^{k-1} \pi_{j,N}^+ \dfrac{A_{j+1}^N}{A_k^N} = \int_0^t  \varphi_+(s) e^{-\int_s^t (\varphi_-+\varphi_+)(u) du} ds
\end{equation}
and 
\begin{equation} \label{eq:preuve:conv:mesinv:5}
\lim_{N \to \infty} \dfrac{A_0^N}{A_{k}^N}= \exp \left( - \int_0^t (\varphi_-+\varphi_+)(s) ds  \right). 
\end{equation}
By \eqref{eq:preuve:conv:mesinv:3}, \eqref{eq:preuve:conv:mesinv:4} and \eqref{eq:preuve:conv:mesinv:5} applied to \eqref{eq:prop:1},  the uniform limit holds:
\begin{align} \label{eq:preuve:conv:mesinv:6}
\lim_{N \to \infty} \mu^N(-1,k) & = e^{- \int_0^t (\varphi_-+\varphi_+)(s) ds} \mu_- (0) + \int_0^t  \varphi_+(s) e^{-\int_s^t (\varphi_-+\varphi_+)(u) du} ds. %\\
%& = \mu_-(t).
\end{align}
The right hand term corresponds to the expression of $\mu_-(t)$, see \eqref{eq:loi_inva:01}.\hfill{$\Box$}
\begin{prop} \label{prop:convCM}
The process $(Z^N(t),\, t \mod T)_{t \geq 0}$, defined by \eqref{eq:def:processuscontinu}, converges in distribution towards the Markov chain $(X_t,\, t \mod T)_{t \geq 0}$.
\end{prop}

\begin{proof}
The proof is divided into three parts. First we prove the convergence of the conditional distribution of successive jumps. In the second part, we prove the convergence of the process on any bounded interval. Finally, we point out that $Z^N$ converges in distribution  to $X$.\\
\textbf{Step 1.} \emph{Convergence of successive jumps}.
We set $Z^N(0)=Z_0=-1$.
%Let $T_1$ be the first transition time from state $-1$ to state $1$ of the Markov chain $(X_t)_{t \geq 0}$. Let $t \geq 0$, we define the integer $k$ by
%$$ k =  \lfloor \frac{t}{T}N \rfloor.$$
%We have
%\begin{equation}
% \mathbb{P}\left(Z_1^N=Z_2^N = \ldots = Z_k^N=-1 \right)=\mathbb{P} \left( Z^N(s)=-1, \forall s \leq t \right).
%\end{equation}
%We want to show
%\begin{align} 
%\nonumber
% \lim_{N \to \infty} \mathbb{P}(Z_1^N=Z_2^N = \ldots = Z_k^N=-1) & = \exp \left( -\int_0^t \varphi_-(s) ds \right) \\ 
% & = \mathbb{P} \left( T_1 \geq t | Z_0=-1 \right). 
%\end{align}
%The probability to stay in the state $-1$ from time $1$ to time $k$ is given by: 
%$$ \mathbb{P}(Z_1^N=Z_2^N  = \ldots = Z_k^N=-1)  = \prod_{j=0}^k \left( 1-\dfrac{T}{N} \varphi_- \left( \dfrac{jT}{N} \right) \right) .$$
%In a same way as in the proof of Lemma \ref{lem:convA}, we obtain
%\begin{equation}
% \lim_{N \to \infty} \mathbb{P}(Z_1^N=Z_2^N  = \ldots = Z_k^N=-1) = \exp \left( -\int_0^t \varphi_-(s) ds \right)
%\end{equation}
We set $T_0^N=T_0=0$ and we define the successive transition times for the chain $(Z^N(t))_{t \geq 0}$ as
$$T_n^N = \inf \{ t > T_{n-1}^N:\  Z^N(t) \neq Z^N(T_{n-1}^N)  \},$$
and for the chain $(X_t)_{t \geq 0}$ the associated transition times are denoted by $T_n$. 
%$$T_n = \inf \{ t > T_{n-1} , X_t \neq X_{T_{n-1}}  \}.$$
For $s \geq 0$ (resp. $t\ge 0$), we set $i=\lfloor \frac{s}{T}N \rfloor$ (resp. $k =  \lfloor \frac{t}{T}N \rfloor$).
We study the transition times from $-1$ to $1$: by \eqref{eq:hypo} we get
\begin{align*}
R_{i,k}&:= \mathbb{P}(T_{2n+1}^N-T_{2n}^N > k | T_{2n}^N =i )   
=  \mathbb{P} \left( Z_{T_{2n}^N+1}^N = \ldots = Z_{T_{2n}^N+k}^N = -1 \Big| T_{2n}^N =i  \right) \\
& =\prod_{j=i}^{i+k-1}(1-\pi_{j,N}^-) = \prod_{j=i}^{i+k-1} \left(1- \dfrac{T}{N} \varphi_-\left( \dfrac{jT}{N} \right) \right).
 \end{align*}
By similar arguments as those developed in Lemma \ref{lem:convA}, we obtain 
\begin{align} \label{eq:convTn:2}
 \lim_{N \to \infty} R_{i,k}& = \exp \left( -\int_s^{s+t} \varphi_-(u) du \right)=  \mathbb{P}(T_{2n+1}-T_{2n} > t | T_{2n}=s). 
\end{align}
For transitions from $1$ to $-1$, that is typically $T_{2n}^N-T_{2n-1}^N$ given $T_{2n-1}^N$, we obtain the same result, just replacing $\varphi_-$ by $\varphi_+$.\\
\textbf{Step 2}. Let us prove now the convergence of the Markov chain on any compact set. Wet set the time interval $[0,1]$, it is straightforward to generalize to any compact set.
Let us define  $\mathcal{N}^N(t)$ the counting process of all transitions of the chain $(Z^N(s))_{s \geq 0}$ on the interval $[0,t]$:
\begin{equation}
 \mathcal{N}^N(t) = \sum_{k \geq 1} \mathbf{1}_{\{T_k^N \leq t\}}, \quad t \in [0,1],
\end{equation}
and  $\mathcal{N}$ the counting process associated with the chain $(X_t)_{t \geq 0}$.
We prove that the distribution sequence $(\mathbb{P}_N)_{N \geq 1}$ of the processes $\mathcal{N}^N$ is tight (see for instance Theorem 13.3 p.141 in \cite{billing}). The theorem requires two conditions. The first one is 
\begin{equation}
\label{eq:cond1:bill}
\lim_{\delta\to 1}\mathbb{P}(\vert X_{\delta}-X_1\vert >\epsilon)=0,\quad \mbox{for any}\quad \epsilon>0.
\end{equation}
Let us define $d_{\delta}=\mathbb{P} \left( \forall t \in [1-\delta, 1], X_t=X_1 \right)$. Then
\begin{align*}
d_\delta & = \sum_{i=\pm}\mathbb{P} \left( \forall t \in [1-\delta, 1], X_t= i1  \right) = \sum_{i=\pm}\mathbb{P} ( X_{1-\delta}=i1 ) \exp \left( - \int_{1-\delta}^1 \varphi_i(s) ds \right) \\
 & \geq  \exp \Big( -\delta \max \Big(\sup_{t \in [0,T]} \varphi_+(t),\sup_{t \in [0,T]} \varphi_-(t) \Big) \Big).
\end{align*}
Thus we obtain
\[
 \lim_{\delta \to 1 } d_\delta= 1 = 1- \lim_{\delta \to 1 } \mathbb{P} \left( X(1) - X(1-\delta) = 2 \right)
\]
which corresponds to the first condition \eqref{eq:cond1:bill}. For the second condition, let $n \in \mathbb{N}^*$. We have to prove that, for any $\eta>0$, there exist $\delta$ and $N_0\in\mathbb{N}$ such that
\begin{equation}\label{eq:cond:bill2}
\mathbb{P} \left( w''_N(\delta) > \epsilon \right)<\eta\quad\mbox{for}\ N\ge N_0,
\end{equation}
where $w''_N$ is the modulus of continuity defined by 
\[
w''_N(\delta) = \underset{t_2-t_1 \leq \delta}{\sup_{0\le t_1 \leq t \leq t_2\le 1}} \{ |Z^N(t)-Z^N(t_1)| \wedge |Z^N(t_2)-Z^N(t)|  \}.
\]
Let us observe that $w''_N(\delta)$ takes either the value $0$ or the value $2$. Hence, for all $n>1$, we define $\Omega_{n}=\bigcap_{k=0}^n \{ T^N_{k+1}-T^N_k > \delta \}$. Therefore 
\begin{align}\label{eq:mod}
\mathbb{P} \left( w''_N(\delta) > \epsilon \right) & = \mathbb{P} \left( w''_N(\delta) = 2 \right)= \mathbb{P} \left( \bigcup_{k\ge 0}\Big[ \{ T^N_{k+1}-T^N_k\le \delta \}\cap\{ T^N_{k+1}\le 1 \} \Big]  \right)\nonumber \\
 %& = 1 - \mathbb{P} \left( \bigcap_{k\ge 0} \Big[ \{ T_{k+1}-T_k> \delta \}\cup \{ \mathcal{N}^N(1)\le k \}  \Big] \right) \\
 & = 1 - \mathbb{P} (\Omega_{\mathcal{N}^N(1)})\le  1 - \mathbb{P} \left( \Omega_{\mathcal{N}^N(1)}\cap\{ \mathcal{N}^N(1)\le n \} \right)\nonumber\\
 &\le 1 - \mathbb{P} \left( \Omega_{n}\cap\{ \mathcal{N}^N(1)\le n \} \right).
\end{align}
We set $\displaystyle\mathbb{P}(\Omega_n)= \sum_{s \in\mathbb{N};\ i=\pm 1 } \rho_{i,s}$ where $\rho_{i,s}$ is defined by the following probability
\begin{align*}
\mathbb{P} \Big(T^N_{n+1}-T^N_n > \delta \Big| T^N_n=s, Z^N(T^N_n) =i\Big)
\mathbb{P} \left(\Omega_{n-1} \cap \{ T^N_n=s, Z^N(T^N_n)=i \}   \right).
\end{align*}
Introducing $\displaystyle\mathcal{S}_{T}=\max_{i=\pm 1} \sup_{t \in [0,T]} \varphi_{i}(t)$, we obtain the lower-bound 
\begin{align*}
\rho_{i,s}& = \prod_{j=s}^{s+\lfloor \frac{\delta N}{T}  \rfloor} \left(1- \dfrac{T}{N} \varphi_{i}\left( \dfrac{jT}{N} \right) \right)
 \mathbb{P} \left( \Omega_{n-1} \cap \{ T^N_n=s, Z^N(T^N_n)=i\} \right) \\
& \geq \left(1- \dfrac{T}{N} \mathcal{S}_{T} \right)^{\lfloor \frac{\delta N}{T}  \rfloor}
 \mathbb{P} \left( \Omega_{n-1}\cap \{ T^N_n=s, Z^N(T_n^N)=i \} \right).
 \end{align*}
 Consequently
 \begin{align*}
\mathbb{P}(\Omega_n)& \geq  \left(1- \dfrac{T}{N} \mathcal{S}_{T} \right)^{\lfloor \frac{\delta N}{T}  \rfloor}
 \mathbb{P} ( \Omega_{n-1}),
 \end{align*}
 and, by induction, the following lower-bound holds 
\begin{equation}\label{eq:minoromeg}
 \liminf_{N\to\infty}\mathbb{P} ( \Omega_n) \geq  \lim_{N\to\infty}\left(1- \dfrac{T}{N} \mathcal{S}_{T}\right)^{n \lfloor \frac{\delta N}{T}  \rfloor}=e^{-n\delta \mathcal{S}_{T} }.
\end{equation}
Let us now use the bounds \eqref{eq:minoromeg} and \eqref{eq:mod} in order to prove \eqref{eq:cond:bill2}. Since  $\mathcal{N}^N(1)$ converges in distribution to $\mathcal{N}(1)$ (see Step 1) as $N\to\infty$, there exist $N_0$ and $n$ such that 
\[
\mathbb{P}(\mathcal{N}^N(1)>n)<\eta/2,\quad \mbox{for all } \ N\ge N_0.
\]
Choose $\delta>0$ small enough such that $1-e^{-n\delta \mathcal{S}_{T} }<\eta/2$ permits to deduce $\mathbb{P}(\Omega_n)>1-\eta/2$  from \eqref{eq:minoromeg}.
In particular, there exist $\delta$ and $N_0$ such that $\forall N \geq N_0$,
$$\mathbb{P} \left( \Omega_n\cap\{ \mathcal{N}^N(1)\le n \} \right) \geq 1 - \eta$$
and therefore \eqref{eq:mod} leads to \eqref{eq:cond:bill2}. Both conditions needed in Theorem 13.3 \cite{billing} are satisfied, we conlude that the distribution of the process $\mathcal{N}^N$ is tight. \\
\textbf{Step 3}. The two first steps (convergence of marginals and tightness)  imply that $\mathcal{N}^N$ converges in distribution towards $\mathcal{N}$ as $N$ becomes large. Let us now deduce the convergence of $Z$ towards $X$. It suffices to use the function $\phi: \mathcal{N} \mapsto -\cos (\pi \mathcal{N})$ which is continuous with respect to the Skorohod topology:  
 $$Z^N = \phi(\mathcal{N}^N)\quad \mbox{and}\quad X=\phi(\mathcal{N}).$$
The convergence of $Z^N$ is then immediate. 
\end{proof}

\begin{rem} The statement of Proposition \ref{prop:convCM} can be improved: not only the process $Z^N$ converges to $X$, the couple $(Z^N,\mathcal{N}^N)$ (Markov chain and associated counting process) also converges in distribution to $(X,\mathcal{N})$. 
\end{rem}

Since both the \emph{periodic stationary probability measure} and the distribution of the periodic discrete-time Markov chain converges, we can obtain the large time asymptotic behaviour of the statistics of transitions for the time-continuous Markov chain via the discretization procedure. The main result announced in the introduction is an adaptation of the following statement which is a consequence of Theorem \ref{thm:asymp:disc}.\\
Let us recall that   $\psi^N(\eta,k)$ (resp. $\psi(\eta,t)$) is the moment generating function of the transitions (from state $-1$ to $1$) for the discrete-time Markov chain (resp. the continuous-time one).
\begin{thm} \label{thm:comportasym}
1. For all $t \ge 0$, the statistics of transitions converge in distribution as $N\to\infty$: the moment generating functions satisfy
\begin{equation}
  \psi^N(\eta,\lfloor \dfrac{tN}{T}\rfloor) \underset{N\to\infty}{\longrightarrow} \psi(\eta,t).
\end{equation}
2. The eigenvalues of the matrix $\mathcal{M}=\lim_{N \to \infty} \mathcal{M}^N$ defined in \eqref{eq:matrice-gen-disc} and \eqref{eq:def:prodm} satisfy $\lambda_1 >1 >\lambda_2 >0$ and the largest one $\lambda_1$ is given by
\begin{equation} \label{eq:lambda1}
  \log \lambda_1 = \log(\eta) \int_0^T \varphi_-(s) \mu_-(s) ds =\log(\eta)\mathbb{E}_\mu[\mathcal{N}_T].
\end{equation}
Here $\mu$ denotes the PSPM defined by \eqref{eq:loi_inva:01}. 
The long time asymptotic behaviour of the generating function is given by
 $$\lim_{t \to \infty} \dfrac{\log \psi(t,\eta)}{t} = \dfrac{\log \eta}{T}  \int_0^T \varphi_-(s) \mu_-(s) ds.  $$
% The average number of jump on the interval $[0,T]$ is given by
% \[
%\mathbb{E}_\mu[\mathcal{N}_T]=\int_0^T \varphi_-(s) \mu_-(s) ds.
%\]
\end{thm}

\begin{proof}
By Proposition \ref{prop:convCM}, the process $\mathcal{N}^N$ converges in distribution to $\mathcal{N}$. The convergence of the generating function is an immediate consequence. 
%Then for all $t \ge 0$,
%$$\lim_{N \to \infty} \mathbb{E}[\eta^{\mathcal{N}^N(t)}]=\mathbb{E}[\eta^{\mathcal{N}(t)}].$$
%For all $\epsilon>0$, it exists $N_0 \in \mathbb{N}$, such that
%\begin{equation}
%\forall N \ge N_0, |\psi^N(\eta,\lfloor \dfrac{tN}{T}\rfloor )-\psi(\eta,t)]|\le \epsilon
%\end{equation}
In particular,
\begin{equation}
 \psi^N(\eta,0) \underset{N\to\infty}{\longrightarrow} \psi(\eta,0) \quad \text{and} \quad \psi^N(\eta,N) \underset{N\to\infty}{\longrightarrow} \psi(\eta,T).
\end{equation}
We deduce that the four coefficients of the monodromy matrix $\mathcal{M}^N$ converge to
$$\mathbb{E}[\eta^{\mathcal{N}(T)}1_{\{ X_T=a \}}|X_0=b]\quad\mbox{with}\quad (a,b)\in\{-1,1\}^2.$$
Using the equation
${\bf \Psi}(\eta,T) = \mathcal{M} {\bf \Psi}(\eta,0)$, the limit matrix $\mathcal{M}$ is in fact the monodromy matrix of the ode  
\eqref{eq:dif:gene}. Both eigenvalues of $\mathcal{M}^N$ converge to the eigenvalues of $\mathcal{M}$. Hence,
according to equation \eqref{eq:prop}
$$\log \lambda_1^N = \log(\eta)\sum_{k=0}^{N-1}\pi_{k,N}^-\mu^N(-1,k)$$
which converges to
$$\log \lambda_1 = \log(\eta) \int_0^T \varphi_-(s) \mu_-(s) ds.$$
Since $\lambda_1>1>\lambda_2$, the inequality
$\lambda_1 \ge 1 \ge \lambda_2$ holds.
By similar arguments as those presented in the proof of Corollary \ref{prop:valp}, we obtain $\lambda_1 > 1 > \lambda_2$ and we can write the moment generating function in the Floquet basis as follows: $\Psi(\eta,t) = r_1 u_1(t) + r_2 u_2(t)$ with $r_1 \neq 0$. Hence $\Psi(\eta,kT) = \lambda_1^k r_1 u_1(0) + \lambda_2^k r_2 u_2(0)$. In order to obtain the large time asymptotics, let us note that, for any $t\ge 0$, there exists $ k \in \mathbb{N}$ such that  $k T \le t < \left( k + 1 \right) T$ and consequently
$$\dfrac{\log \Psi(\eta,kT)}{(k+1)T} \le \dfrac{\log \Psi(\eta,t)}{t} \le \dfrac{\log \Psi(\eta,(k+1)T)}{kT}.$$
Both bounds tend to $\log \lambda_1/T$.\\
Finally let us note that $\log\lambda_1^N$ can easily be expressed as the mean number of transitions during one period and starting with the PSPM. This identity remains true in the large $N$ limit. Indeed it suffices to use the convergence of the generating functions to deduce that the family of random variables $(\mathcal{N}^N_N)_{N\ge 0}$ is uniformly integrable. Therefore applying Vall\'ee-Poussin's theorem (see for instance Theorem T22 in \cite{meyer}) to the function $t\mapsto \eta^t$ with $\eta>1$. We thus obtain the convergence of the average number of transitions starting from any initial distribution and in particular starting from the PSPM. So we deduce:
\[
\mathbb{E}_\mu[\mathcal{N}_T]=\int_0^T \varphi_-(s) \mu_-(s) ds.
\]
\end{proof}

%\begin{rem}
 %A un multiplicateur de Floquet $\lambda$, on associe l'exposant de Floquet $\mu$ tel que $\lambda=e^{\mu T}$. On a alors
 %$$\mu = \dfrac{\log \lambda_1}{T}.$$
%\end{rem}

\section{Two examples in the stochastic resonance framework}\label{sec:resonance}
We seek to describe the phenomenon of stochastic resonance. The continuous-time Markov chain $X_t$ oscillates between two values $\pm 1$ according to a T-periodic infinitesimal generator $\mathcal{Q}_t$. Then by varying the period, we observe that the behaviour of the chain changes and adopts \emph{more or less} periodic paths. The aim in each example is to find the optimal period such that the behaviour of the paths looks like the most periodic as possible. That's why we shall introduce a criterion which measures the periodicity of any random path. We propose to use a criterion associated with the largest Floquet exponent of the generating function. The interesting tunings correspond to situations where this exponent is close to the value $\log(\eta)$. Such a criterion was already presented in \cite{Talkner}.
\subsection{An infinitesimal generator constant on each half period}
In this first example, we consider T-periodic rates given by
\begin{equation} \label{eq:defdephi}
\varphi_-(t)=\varphi_0 1_{\{0\le t <T/2\}}+\varphi_1 1_{\{T/2\le t< T\}}=\varphi_0+\varphi_1-\varphi_{+}(t).
\end{equation}
where $\varphi_0=p\, e^{-\frac{V}{\epsilon}}$ et $\varphi_1=q\, e^{-\frac{v}{\epsilon}}$, $v<V$. This Markov model is often used in the stochastic resonance framework (see for instance \cite{pav02}).
Here we can compute explicitly the invariant measure (see also \cite{pav02} Proposition 4.1.2 p.34)
\begin{lemma}\label{lem:exemp:stat} The periodic stationary probability measure PSPM is given by:
\begin{equation}\label{lem:exemp}
\mu_-(t)=\frac{e^{-(\varphi_0+\varphi_1)t}}{1+e^{-(\varphi_0+\varphi_1)T/2}}\frac{\varphi_0-\varphi_1}{\varphi_0+\varphi_1}+\frac{\varphi_1}{\varphi_0+\varphi_1}\end{equation}
and $\mu_-(t)+\mu_+(t)=1$, $\mu_\pm(t+T/2)=\mu_\mp(t)$.
\end{lemma}
\begin{proof}
Using the description of the PSPM in Proposition \ref{prop:largetime} we obtain 
\begin{align}\label{eq:expressinv}
\mu_-(t)&=\mu_-(0)e^{-(\varphi_0+\varphi_1)t}+\frac{\varphi_1}{\varphi_0+\varphi_1}\Big( 1-e^{-(\varphi_0+\varphi_1)t} \Big)\nonumber\\
&=\left(\mu_-(0)-\frac{\varphi_1}{\varphi_0+\varphi_1}\right)\,e^{-(\varphi_0+\varphi_1)t}+\frac{\varphi_1}{\varphi_0+\varphi_1},\quad 0\le t<T/2.
\end{align}
Furthermore, by symmetry arguments, the dynamics of the periodic invariant measure satisfies:
\(
\mu_\pm(t+T/2)=\mu_\mp(t) \quad \mbox{for all}\ t\ge 0.
\)
We deduce in particular that $\mu_-(T/2)=\mu_+(0)=1-\mu_-(0)$. 
%Thus
%\[
%1-\mu_-(0)=\left(\mu_-(0)-\frac{\varphi_1}{\varphi_0+\varphi_1}\right)\,e^{-(\varphi_0+\varphi_1)T/2}+\frac{\varphi_1}{\varphi_0+\varphi_1}.
%\]
Thus 
\[
\mu_-(0)=\frac{\varphi_0+\varphi_1\,e^{-(\varphi_0+\varphi_1)T/2}}{(\varphi_0+\varphi_1)(1+e^{-(\varphi_0+\varphi_1)T/2})}
\]
The equation \eqref{eq:expressinv} then permits to conclude.
\end{proof}
An immediate consequence of Theorem \ref{thm:comportasym} and Lemma \ref{lem:exemp:stat} leads to the explicit computation of the largest Floquet exponent (the details of the proof are left to the reader). 
\begin{prop}\label{prop:exemp1}
The largest Floquet exponent of the ode \eqref{eq:dif:gene} with the rates \eqref{eq:defdephi}, which corresponds to the asymptotic behavior of the generating function of the statistics of transitions $\mathcal{N}_t$, is given by $\log(\eta)\mathbb{E}_\mu[\mathcal{N}_T]$ where
\begin{equation}\label{eq:prop:exemp1}
\mathbb{E}_\mu[\mathcal{N}_T]%=\log(\eta)\int_0^T\varphi_-(t)\mu_-(t)dt
= \frac{\varphi_0\varphi_1 T}{\varphi_0+\varphi_1} +\left( \frac{\varphi_0-\varphi_1}{\varphi_0+\varphi_1} \right)^2\tanh\Big( (\varphi_0+\varphi_1)T/4 \Big).
\end{equation}
\end{prop}
%\begin{proof}
%According to  theorem \ref{thm:comportasym}, asymptotic behavior is liked to $\mathbb{E}_\mu[\mathcal{N}_T]$, that is to say
%\[
%I=\int_0^T\varphi_-(t)\mu_-(t)dt
%\]
%We just have to use \eqref{eq:defdephi} and \eqref{lem:exemp} to obtain:
%\begin{align*}
%I&=\int_{0}^{T/2}\varphi_-(t)\mu_-(t)dt+\int_{0}^{T/2}\varphi_+(t)\mu_+(t)dt\\
%&=\varphi_0\int_{0}^{T/2}\mu_-(t)dt+\varphi_1\int_{0}^{T/2}\mu_+(t)dt=\frac{\varphi_1 T}{2}+(\varphi_0-\varphi_1)\int_{0}^{T/2}\mu_-(t)dt\\
%&=\frac{\varphi_1 T}{2}+\frac{\varphi_0-\varphi_1}{\varphi_0+\varphi_1}\left\{ \varphi_1\,\frac{T}{2}+\frac{\varphi_0-\varphi_1}{\varphi_0+\varphi_1}\frac{1-e^{-(\varphi_0+\varphi_1)T/2}}{1+e^{-(\varphi_0+\varphi_1)T/2}}  \right\}.
%\end{align*}
%Expression \eqref{eq:prop:exemp1} is easily deduced.
%\end{proof}
We are interested in the phenomenon of stochastic resonance associated to continuous-time process $(X_t,\, t\ge 0)$. This process essentially depends on two parameters: a parameter $\epsilon$ describing the intensity of the transition rates between both states $\{-1,+1\}$ (some small $\epsilon$ corresponds to a \emph{frozen} situation: the Markov chain remains in the same state for a long while) and a second parameter $T$, the period of the process dynamics. By considering the normalized process $Y_t=X_{tT}$, especially its paths  on a fixed interval $[0,S]$, we observe the following phenomenon (for fixed $\epsilon$):
if $T$ is small then there are very few transitions of $Y$: the process tends to remain in its original state. If $T$ is large, $Y$ behaves in a chaotic way: lots of  transitions are observed. For some intermediate values of $T$, the random paths of $Y$ are close to deterministic periodic functions (one transition in each direction per period). Let us note that this phenomenon can also be observed by freezing the period length $T$ and varying the intensity $\epsilon$ of the rates.

The aim is therefore to point out the best relationship (tuning) between $\epsilon$ and $T$ which makes the process $Y$ the most periodic as possible. If the process is close to a periodic function then the number of transition from state $-1$ to sate $+1$ is close to $1$ per period, which leads to find the tuning corresponding to the Floquet exponent equal to $\log\eta$. By Proposition \ref{prop:exemp1}, it is then sufficient to find the best relation between $\epsilon$ and $T$ such that 
\begin{equation}\label{eq:tuning}
\mathbb{E}_\mu[\mathcal{N}_T]=1.
\end{equation}
\begin{figure}[h]\label{fig}
\centerline{\includegraphics[scale=0.5,angle=0]{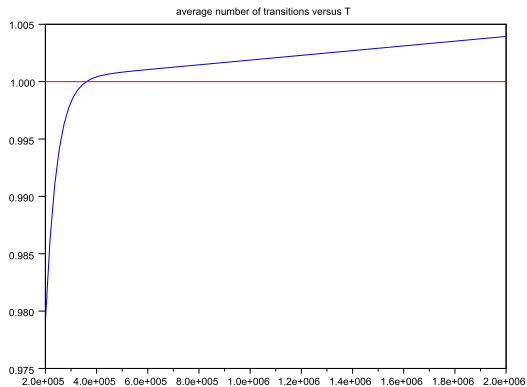}}
\caption{Average number of transitions}
\end{figure}
In Figure \ref{fig}, we set $\epsilon=0.1$, $V=2$, $v=1$, and let $T$ vary. We compute numerically the average number of transitions per period. We can clearly observe that there is one and only one period corresponding to the condition \eqref{eq:tuning}. 
%\end{minipage}\\
\begin{prop} \label{prop:exemp1opt} Let $T^\epsilon_{\rm opt}$ be the period which provides an average number of transitions per period equal to $1$. The following asymptotic behaviour holds, as $\epsilon$ tends to $0$,
\begin{equation}\label{eq:exemp1opt}
T^ \eps_{\rm opt}\sim \frac{V-v}{2q\epsilon}\, e^{v/\epsilon}.
\end{equation}
\end{prop}
\begin{proof}
The condition \eqref{eq:tuning} combined with Proposition \ref{prop:exemp1} leads to the equation 
\[
\frac{\varphi_0\varphi_1 T}{\varphi_0+\varphi_1} +\left( \frac{\varphi_0-\varphi_1}{\varphi_0+\varphi_1} \right)^2\tanh\Big( (\varphi_0+\varphi_1)T/4 \Big)=1.
\]
The aim is to solve it and let $\epsilon$ tend to $0$. The left member in the previous equation is an increasing function of $T$. We introduce the change of variable $U^\epsilon=(\varphi_0+\varphi_1)T/4$. We first prove that $U^\epsilon$ increases as $\epsilon$ decreases. $U^\epsilon$ satisfies $K(U^\epsilon,\epsilon)=1$ with
\[
K(U^\epsilon,\epsilon)=\frac{4\varphi_0\varphi_1 U^\epsilon}{(\varphi_0+\varphi_1)^2} +\left( \frac{\varphi_0-\varphi_1}{\varphi_0+\varphi_1} \right)^2\tanh(U^\epsilon).
\]
Both functions $\epsilon\mapsto K(\cdot,\epsilon)$ and $x\mapsto K(x,\cdot)$ decrease for $\epsilon$ small and $x$ large enough. It follows that $U^\epsilon$ increases when $\epsilon$ decreases and tends to $0$. Let us assume $U^\epsilon\to U_0<\infty$ in the limit $\epsilon\to 0$. Then $K(U^\epsilon,\epsilon)\to \tanh(U_0)$ which contradicts the identity $K(U^\epsilon,\epsilon)=1$. We  deduce that $U^\epsilon\to\infty$ when $\epsilon\to 0$. Now let us set $t=e^{-V/\epsilon}$ and $\beta=v/V<1$. With these new parameters $K$ can be written like
\[
1=\tilde{K}(U,t)= \frac{4pqt^{1+\beta}U^\epsilon}{(pt+qt^\beta)^2}+\left( \frac{pt-qt^\beta}{pt+qt^\beta} \right)^2\tanh(U^\epsilon).
\]
Let us define $\tanh(U^\epsilon)=:1-W$ then $U^\epsilon=\frac{1}{2}\, \log\left(\frac{2-W}{W}\right)$, we obtain that $W$ tends to $0$ when $t\to 0$ and the previous equation becomes:
\[
1=\hat{K}(W,t)=\frac{4pqt^{1+\beta}}{(pt+qt^\beta)^2}\, \log\left( \frac{2-W}{W} \right)+\left( \frac{pt-qt^\beta}{pt+qt^\beta} \right)^2(1-W).
\]
Thus, when $t\to 0$, we have
\begin{align*}
\hat{K}(W,t)-1&=\frac{-2pqt^{1+\beta}}{(pt+qt^\beta)^2}\, (\log W+o(\log W))\\
&\quad +(1-\frac{4p}{q}\,t^{1-\beta}+o\left(t^{1-\beta}\right))(1-W)-1\\
&=\frac{-2pqt^{1+\beta}}{(pt+qt^\beta)^2}\, \log W-W+o(t^{1-\beta}\log W)=0.
\end{align*}
If $W=r_0t^\alpha\log(t)R(t)$ with $\alpha=1-\beta$ and $r_0=-\frac{2p\alpha}{q}=-\frac{2p}{q}(1-\beta)$, we obtain the limit $R(t)\to 1$ when $t\to 0$ and therefore
\[
U^\epsilon\sim -\frac{1}{2}\, \log\left(\frac{2p}{q}\, (1-\beta) t^{1-\beta}(-\log t)\right)\sim -\frac{1-\beta}{2}\, \log t \sim \frac{(1-\beta)V}{2\epsilon}=\frac{V-v}{2\epsilon}.
\]
We recall $U^\epsilon=(\varphi_0+\varphi_1)T/4$ which leads to the result set.
\end{proof}
In \cite{pav02}, several quality measures have been proposed to point out the optimal tuning of $Y$: the spectral power amplification (SPA), the SPA to noise intensity ratio (SPN), the energy (En), the energy to noise intensity ratio (ENR), the \emph{out-of-phase} measure which describes the time spent in the most attractive state, the entropy or relative entropy. In his PhD report, I. Pavlyukevich computes for each measure the optimal relation between $\epsilon$ and  $T_{\rm mes}^\epsilon$, the length of the period, in the small $\epsilon$ limit, we adopt a similar procedure in Proposition \ref{prop:exemp1opt}. So we can now gather these quality measures into three families:
\begin{itemize}
\item for the first family, the optimal tuning satisfies $T_{\rm mes}^\epsilon=o(T_{\rm opt}^\epsilon)$ where $T_{\rm opt}^\epsilon$ is given by \eqref{eq:exemp1opt}. The associated Markov chain has an average number of transitions from $-1$ to $+1$ strictly smaller than $1$. This family contains in particular the SPN.
\item The second family concerns $T_{\rm opt}^\epsilon=o(T_{\rm mes}^\epsilon)$. The Markov chain has then more than one transition per period on average. This family contains most of the measures: SPA, En, Out-of-phase, the entropy and relative entropy.
\item Finally in the third family $T_{\rm opt}^\epsilon$ and $T_{\rm mes}^\epsilon$ are comparable, this is namely the case for ENR.
\end{itemize}

\subsection{Infinitesimal generator with constant trace}
Let us finally present a second example of periodic forcing in the stochastic resonance framework. This model was introduced by Eckmann and Thomas \cite{Eckmann}. The aim in this paragraph is to find the optimal tuning between the noise intensity in the system and the period length in order to reach an average number of transitions during one period close to $1$. This approach is different from the study prensented in \cite{Eckmann}.\\ 
The model consists in a continuous-time Markov chain with periodic forcing: the transition rates are given by
\begin{equation} \label{eq:defdephi2}
\varphi_-(t)=\epsilon(a+ \cos \omega t) \quad \text{and} \quad \varphi_+(t)=\epsilon (a- \cos \omega t),\quad a>1.
\end{equation}
The period satisfies
$T=(2\pi)/\omega$. In this particular case, the trace of the infinitesimal generator, defined by \eqref{eq:gen-inf}, is a constant function. It is then quite simple to compute explicitely the \emph{periodic stationary probability measure} and the Floquet multipliers associated with the moment generating function of the statistics of transitions.
\begin{lemma}
 The periodic stationary probability measure of the periodic forced Markov chain is given by  
 \begin{equation} \label{eq:mesinv:ex2}
 \mu_-(t)=\dfrac{1}{2}-\dfrac{\epsilon}{4a^2\epsilon^2+\omega^2}(2a\epsilon \cos \omega t+\omega \sin \omega t).
 \end{equation}
\end{lemma}

\begin{proof}
Using Proposition \ref{prop:largetime}, we obtain
 \begin{align*}
 \mu_-(t) & =\mu_-(0) e^{-2\epsilon a t}+\int_0^t (\epsilon a-\epsilon \cos \omega s) e^{-2\epsilon a(t-s)} ds .
\end{align*} 
% \\
%  & = \mu_-(0) e^{-2at} +\dfrac{1-e^{-2at}}{2}-\epsilon e^{-2at} \int_0^t \cos (\omega s) e^{2as} ds.
% \end{align*}
%Calculating the integral, we have 
%$$\int_0^t \cos (\omega s) e^{2as} ds=\dfrac{-2a}{4a^2+\omega^2} + \dfrac{2a e^{2at}}{4a^2+\omega^2} \cos \omega t + \dfrac{\omega e^{2at}}{4a^2+\omega^2} \sin \omega t$$
%Then we have for the measure the relation
Hence
$$\mu_-(t) =   \mu_-(0) e^{-2\epsilon at} + \dfrac{1-e^{-2\epsilon at}}{2} + \dfrac{2\epsilon^2 a  e^{-2\epsilon at}-2\epsilon^2 a  \cos \omega t-\epsilon \omega \sin \omega t}{4\epsilon^2a^2+\omega^2}.$$
Setting $\mu_-(T)=\mu_-(0)$, we obtain
%$$\mu_-(0)=\mu_-(0)e^{-2aT}+\dfrac{1-e^{-2aT}}{2}+\dfrac{2\epsilon a}{4a^2+\omega^2}(e^{-2aT}-1).$$
%This gives for $\mu_-(0)$ the expression
%\begin{equation} \label{eq:mu0:ex2}
\(
 \mu_-(0) = \dfrac{1}{2}-\dfrac{2a \epsilon^2}{4\epsilon^2a^2+\omega^2}
\)
and consequently the announced statement.
\end{proof}
An application of Theorem \ref{thm:comportasym} permits to describe the large time asymptotics for the moment generating function of the transitions from state $-1$ to state $+1$. It suffices to compute explicitely $I=\int_0^T \varphi_-(t)\mu_-(t) dt$. The result is described in the following statement while the proof is left to the reader.
\begin{prop}
The largest Floquet exponent associated with the statistics of transitions $\mathcal{N}_T$ (the moment generating function) is equal to $\log(\eta) \mathbb{E}_{\mu}[\mathcal{N}_T]$ with
 \begin{equation} \label{eq:nbremoyen2}
  \mathbb{E}_{\mu}[\mathcal{N}_T] = \dfrac{\epsilon aT}{2}-\dfrac{\eps^3aT}{4\epsilon ^2a^2+\omega^2},
 \end{equation}
and $\mu$ given by \eqref{eq:mesinv:ex2}.
\end{prop}

%\begin{proof}
%From theorem \ref{thm:comportasym}, asymptotic behavior is given by
%$$I=\int_0^T \varphi_-(t)\mu_-(t) dt$$
%From the definition of $\varphi_-$ given by \eqref{eq:defdephi2} and the result obtain for $\mu_-$ given by \eqref{eq:mesinv:ex2}, we obtain
%\begin{align*}
% \int_0^T \varphi_-(t) \mu_-(t) dt & = \int_0^T (a+\epsilon \cos \omega t)\left(\dfrac{1}{2}-\dfrac{\epsilon}{4a^2+\omega^2}(2a \cos \omega t+\omega \sin \omega t) \right) dt \\
% & = \int_0^T \dfrac{a}{2} + \dfrac{\epsilon}{2} \cos \omega t - \dfrac{\epsilon^2}{4a^2+\omega^2} (2a \cos^2 \omega t - \omega \cos \omega t \sin \omega t) dt \\
% & = \dfrac{aT}{2}-\dfrac{\epsilon^2 a}{4a^2+\omega^2} \int_0^T (1+\cos 2\omega t) dt
%\end{align*}
%Integral $\int_0^T \cos 2\omega t dt$ is zero and we obtain result \eqref{eq:nbremoyen2} 
%\end{proof}
Let us now discuss the suitable choice of the period such that $\mathbb{E}_{\mu}[\mathcal{N}_T]=1$. We then need to solve
\begin{equation} \label{eq:relation:ex2}
\pi \epsilon a (4\epsilon^2 a^2+\omega^2) - 2\pi \epsilon^3 a = \omega (4\epsilon^2 a^2+\omega^2).
\end{equation}
It is obvious that $\omega$ is of the order $\epsilon$, we set $\omega=\mu\epsilon$ and look for the best choice of the parameter $\mu$. Considering \eqref{eq:relation:ex2}, the optimal value $\mu$ is in fact a real root of the following polynomial function
\[
P(\mu):=\mu^3-\pi a\mu^2+4a^2\mu+2\pi a(1-2a^2)
\]
It is straightforward to prove that this polynomial function has a single positive root since it is  increasing and verifies $P(0)<0$. Using the Cardan formula, we can obtain an explicit expression of $\mu_{\rm optimal}$ which depends of course on the coefficient $a$, this dependence is asymptotically linear as $a$ becomes large.\\[5pt]
{\bf Acknowledgements}\\
We are very grateful to Mihai Gradinaru for interesting conceptual and scientific discussions on the problem of stochastic resonance associated to the two-states Markov chain. His availability was greatly appreciated.
%\section{Notations}
%\begin{tabular}{|l|c|c|c|}
%\hline
%& continuous model  & discrete model \\
%\hline
%Markov chain & $X_t$ & $Z_n^N$ or $Z^N(t)$  \\
%\hline
%probability or transition rate & $\varphi_\pm(t)$ & $\pi_{n,N}^\pm$  \\
%\hline
%infinitesimal generator, transition matrix & $\mathcal{Q}_t$ or $\mathcal{Q}(\eta,t)$ & $M^N_n$  \\
%\hline
%process law at time $t$ & $\nu(t)$ &  \\
%\hline
%invariant measure & $\mu(t)$ or $\mu_\pm(t)$ & $\mu^N(-1,n)$  \\
%\hline
%generating function & $\psi(\eta,t)$ or $\psi_\pm(\eta,t)$ & $\Psi^N(\eta,n)$ or $\Psi^N_\pm(\eta,n)$  \\
%\hline
%transition number between $-1$ and $1$ & $\mathcal{N}_t$ & $\mathcal{N}^N_t$  \\
%\hline
%transition time & $T^N_n$ & $T_n$  \\
%\hline
%\end{tabular}
%

\bibliographystyle{plain}
\bibliography{biblio}

\end{document}